\documentclass[12pt]{amsart}

\usepackage{amssymb,  mathrsfs}
\usepackage{graphicx}

\newtheorem{theorem}{Theorem}
\newtheorem{lemma}[theorem]{Lemma}
\newtheorem{proposition}[theorem]{Proposition}
\newtheorem{observation}[theorem]{Observation}

\newtheorem{them}{Theorem}
\newtheorem{lema}[them]{Lemma}

\newtheorem{definition}[theorem]{Definition}
\newtheorem{example}[theorem]{Example}

\begin{document}

\title[Excellent graphs ]  { Excellent graphs with respect to domination:  subgraphs induced by minimum dominating sets}

\author[]{Vladimir Samodivkin}
\address{Department of Mathematics, UACEG, Sofia, Bulgaria}
\email{vl.samodivkin@gmail.com}
\today

\begin{abstract}
A graph $G=(V,E)$ is $\gamma$-excellent if $V$ is a union of all 
$\gamma$-sets of $G$, where $\gamma$ stands for the domination number. 
Let $\mathcal{I}$ be a set of all mutually nonisomorphic graphs and $\emptyset \not= \mathcal{H} \subsetneq \mathcal{I}$. 
In this paper we initiate the study of the $\mathcal{H}$-$\gamma$-excellent graphs, 
which we define as follows.
 A   graph $G$ is   $\mathcal{H}$-$\gamma$-excellent  if the following hold: 
(i)  for every $H \in \mathcal{H}$ and for each $x \in V(G)$ 
									there exists an induced subgraph $H_x$ of $G$ 
                   such that $H$ and $H_x$ are isomorphic,  $x \in V(H_x)$ and        
									$V(H_x)$ is a subset of some  $\gamma$-set of $G$, and 																	
(b)            the vertex set of  every  induced subgraph $H$  of $G$,
										which is isomorphic to some element of $\mathcal{H}$, 
                     is  a subset of some $\gamma$-set of $G$.   
	For each of some well known   graphs, 
including cycles, trees and some cartesian products of two graphs, 
   we describe its  largest set $\mathcal{H} \subsetneq \mathcal{I}$
 for which the graph  is   $\mathcal{H}$-$\gamma$-excellent.
Results on $\gamma$-excellent regular graphs and 
a generalized lexicographic product of  graphs are presented. 
Several open problems and questions are posed. 
\end{abstract}

\maketitle

%\linenumbers

\section{Introduction}

All graphs in this paper will be finite, simple, and undirected. 
We use \cite{hhs1} as a reference for terminology and notation which are not explicitly defined here. 
For a graph $G = (V(G), E(G))$, let $\pi$ be a graphical property that can
be possessed, or satisfied by the subsets of $V$.
 For example, being a maximal complete subgraph, a maximal independent set, 
 acyclic,  a closed/open neighborhood, a minimal dominating set, etc.
Suppose that $f_\pi$   and $F_\pi$ are  the associated  graph invariants: 
 the minimum  and maximum cardinalities of a set with property $\pi$.
Let $\mu \in \{f_\pi, F_\pi\}$. For a graph $G$, denote by $\mathtt{M}_\mu(G)$ 
the family of all  subsets of $V(G)$ each of which has property $\pi$ and cardinality $\mu(G)$. 
Each element of $\mathtt{M}_\mu(G)$ is called  a $\mu$-{\em set} of $G$. 
Fricke et al. \cite{fhhhl} define a graph $G$ to be $\mu$-{\em excellent} if 
each its vertex belongs to some $\mu$-set. 
Perhaps historically the first results on $\mu$-excellent graphs 
were published by  Berge \cite{berge}                                                                                    
who  introduced the class of $B$-graphs consisting of all  
graphs in which every vertex is  in a maximum independent set. 
Of course  all $B$-graphs form  the class of $\beta_0$-excellent graphs, 
where $\beta_0$ stand for the  independence number. 
The study of excellent graphs with respect to the  some domination related parameters 
was initiated by   Fricke et al. \cite{fhhhl} and continued e.g. in \cite{bs,hh,h2,msu,sam3,sry,ys}.

 In this paper  we focus on the following  subclass of the class of $\mu$-excellent graphs. 

\begin{definition} \label{def1}
Let $\mathcal{I}$ be a set of all mutually nonisomorphic graphs and $\emptyset \not= \mathcal{H} \subsetneq \mathcal{I}$. 
We say that a   graph $G$ is   $\mathcal{H}$-$\mu$-excellent  if the following hold: 
\begin{itemize}
\item[(i)]  For each $H \in \mathcal{H}$ and for each $x \in V(G)$ 
									there exists an induced subgraph $H_x$ of $G$ 
                   such that $H$ and $H_x$ are isomorphic,  $x \in V(H_x)$ and        
									$V(H_x)$ is a subset of some  $\mu$-set of $G$.																		
\item[(ii)] For each  induced subgraph $H$  of $G$,
										which is isomorphic to some element of $\mathcal{H}$, 
                     there is  a $\mu$-set of $G$ having $V(H)$ as a subset.   
\end{itemize} 
\end{definition}

By the above definition it immediately follows that each  $\mathcal{H} $-$\mu$-excellent graph 
is $\mu$-excellent. If a graph $G$ is  $\mathcal{H} $-$\mu$-excellent 
and $\mathcal{H}$ contains only one element, 
e.g.   $\mathcal{H} = \{H\}$, we sometimes omit the brackets 
and say that a graph $G$ is $H$-$\mu$-excellent. 
Define the $\mu$-{\em excellent family of induced subgraphs} of a $\mu$-excellent graph $G$, 
denoted by $G\left\langle \mu \right\rangle$, as the family of all graphs 
$H \in \mathcal{I}$ for which $G$ is $H$-$\mu$-excellent. 
The next  two observations are  obvious.

\begin{observation}  \label{mu}
If $G$ is a $\mu$-excellent  graph, then $\{K_1\} \subseteq G\left\langle \mu \right\rangle$  
 and  $\mu(G)  \geq \max \{|V(H)| \mid H \in G\left\langle \mu \right\rangle\} $.  
\end{observation}

\begin{observation}  \label{munu}
Let a graph   $G$ be both $\mu$-excellent  and $\nu$-excellent. 
If the set of all  $\mu$-sets and the set of all $\nu$-sets of $G$ coincide, 
then $G\left\langle \mu \right\rangle  =  G\left\langle \nu \right\rangle$.
\end{observation}

As first examples of $\mathcal{H}$-$\mu$-excellent graphs let us consider the case $\mu = \beta_0$.  
Clearly, any $\beta_0$-excellent graph $G$ is $\{\overline{K_1}, \overline{K_{\beta_0(G)}}\}$-$\beta_0$-excellent.
A graph is $r$-{\em extendable} if every independent set of size $r$ 
is contained in a maximum independent set (Dean and Zito \cite{dz}). 
Clearly,   a graph  is $\{\overline{K_1},\overline{K_2},..,\overline{K_r}\}$-$\beta_0$-excellent  
if and only if  it  is $s$-extendable for all $s=1,2,..,r$.    
 Plummer \cite{p} define a graph $G$ to be {\em well covered} whenever 
$G$ is $k$-extendable for every integer $k$. In other words, 
a graph $G$ is well covered if and only if 
$G\left\langle \beta_0 \right\rangle = \{\overline{K_1},\overline{K_2},..,\overline{K_{\beta_0(G)}}\}$.

In this paper we concentrate mainly on excellent graphs with respect to the domination number $\gamma$.  
We give basic terminologies and notations in the rest of this section. 
In Section $2$ we describe the $\gamma$-excellent
 family of induced subgraphs for some well known graphs.
 In Section $3$ we show  that,  under appropriate restrictions,  
the generalized lexicographic product of graphs has 
the same excellent family of induced subgraphs 
with respect to six domination-related parameters. 
Section 4   contains results on $\gamma$-excellent regular graphs and trees.   
  We conclude in Section 5 with some open problems.

In a graph $G$, for a subset $S \subseteq V (G)$ the {\em subgraph induced} by $S$ is the graph
$\left\langle S \right\rangle$ with vertex set $S$ and  
two vertices in $\left\langle S \right\rangle$ are adjacent if and only if they are adjacent in $G$.
The {\em complement} $\overline{G}$ of $G$ is the graph whose
vertex set is $V (G)$ and whose edges are the pairs of nonadjacent vertices of $G$.
We write $K_n$ for the {\em complete graph} of order $n$ and $P_n$ for the  {\em path} 
on $n$ vertrices. Let $C_m$ denote the {\em cycle} of length $m$.
 For any vertex $x$ of a graph $G$, $N_G(x)$ denotes the set of all neighbors 
of $x$ in $G$, $N_G[x] = N_G(x) \cup \{x\}$  and the degree of $x$ is $deg_G(x) = |N_G(x)|$. 
The {\em minimum} and {\em maximum} degrees
 of a graph $G$ are denoted by $\delta(G)$ and $\Delta(G)$, respectively.
For a subset $S \subseteq V (G)$, let $N_G[S] = \cup_{v \in S}N_G[v]$. 
Let $X \subseteq V(G)$ and $x \in X$. 
The $X$-{\em private neighborhood} of $x$, denoted by $pn_G[x,X]$ or simply by $pn[x,X]$ 
if the graph is clear from the context, is the set $\{y \in V(G) \mid N[y] \cap X = \{x\}\}$. 
A {\em leaf} is a vertex of degree one and a {\em support vertex} is a vertex adjacent to a leaf. 
The $1$-{\em corona}, denoted $cor(U)$, of a graph $U$ is the graph obtained from $U$
 by adding a degree-one neighbor to every vertex of $U$.
An {\em isomorphism} of graphs $G$ and $H$ is a bijection between the vertex sets of $G$ and $H$
$f \colon V(G)\to V(H)$ such that any two vertices $u$ and $v$ of $G$
 are adjacent in $G$ if and only if $f(u)$ and $f(v)$ are adjacent in $H$. 
If an isomorphism exists between two graphs, then the graphs are called isomorphic
 and denoted as $ G\simeq H$. We use the notation $[k]$  for $\{1,2,..,k\}$.

An  {\em independent set}  is a set of vertices in a graph, no two of which are adjacent. 
The {\em independence number} of $G$, denoted $\beta_0(G)$, is the maximum size of an independent set in $G$. 
The {\em independent domination number} of $G$, denoted by $i(G)$, is the minimum size of a maximal independent set in $G$. 
A subset $D \subseteq V(G)$ is called a {\em dominating set} (or a {\em total dominating set}) in $G$, if
for each $x \in  V (G) -D$ (or for each $x \in V (G)$, respectively) there exists a vertex $y \in D$ adjacent to $x$.  
A  dominating set $R$ of a graph $G$  is a {\em restrained dominating set} (or an {\em outer-connected dominating set}) in $G$,    if
every vertex  in $V(G)-R$ is adjacent to a vertex in $V(G)-R$ (or $V(G)-R$ induces a connected graph, respectively). 
The minimum number of vertices of a dominating set in a graph $G$ is  the {\em domination number} $\gamma(G)$ of $G$.  
 Analogously the {\em total domination number} $\gamma_t(G)$,  the {\em restrained domination number} $\gamma_r(G)$  and 
the {\em outer-connected domination number} $\gamma^{oc}(G)$ are defined. 
The minimum cardinality of a set $S$ which is simultaneously total dominating and restrained dominating  in $G$ 
 is called  the  {\em total restrained domination number} $\gamma_{tr}(G)$ of $G$. 
The minimum cardinality of a set $S$ which is simultaneously total dominating and outer-connected dominating  in $G$ 
 is called  the  {\em total outer-connected domination number} $\gamma_{t}^{oc}(G)$ of $G$.

\section{\sc Examples}

Here we find the $\gamma$-excellent family of induced subgraphs of some well known graphs.

\begin{example}\label{e12}  
Let $G$ be a connected graph with $\gamma(G)=2$.
In \cite{jay} it is proved that (in our terminology)  $G$ is $K_2$-$\gamma$-excellent if and only if 
$G$ is a complete $r$-partite graph $K_{n_1,n_2,..,n_r}$, $n_i \geq 2$, $i=1,2,..,r \geq 2$.  
Clearly $K_{2,2,..,2}\left\langle \gamma \right\rangle = \{K_1,K_2, \overline{K_2}\}$ and 
$K_{n_1,n_2,..,n_r}\left\langle \gamma \right\rangle = \{K_1,K_2\}$ when $n_s \geq 3$ for some $s \in [r]$. 
\end{example}

\begin{example}\label{pncr}
Let  $\nu \in \{\gamma, i\}$. Then all the following hold:
\begin{itemize}
\item[(i)] (folklore)  $\nu(P_n) = \left\lceil n/3 \right\rceil$ and $\nu(C_r) = \left\lceil r/3 \right\rceil$. 
                                    $C_r$ is $\nu$-excelent for all $r \geq 3$. 
                                       $P_n$ is $\nu$-excellent if and only if $n=2$ or $n \equiv 1 \pmod 3$. 
\item[(ii)]  $P_n\left\langle \nu \right\rangle = \{K_1\}$ when $n \in \{1,2\} \cup \{7,10,\dots\}$ 
                     and $P_4\left\langle \nu \right\rangle = \{K_1,\overline{K_2}\}$
\item[(iii)] $C_5\left\langle \nu \right\rangle = \{K_1, \overline{K_2}\}$ and 
                     $C_{3r}\left\langle \nu \right\rangle =  C_{5+3r}\left\langle \nu \right\rangle = \{K_1\}$, $r \geq 1$.
\item[(iv)]  $C_7\left\langle \gamma \right\rangle = \{K_1, K_2,\overline{K_2},\overline{K_3}\}$, 
                     and $C_{3r+1}\left\langle \gamma \right\rangle =\{K_1, K_2, \overline{K_2}\}$ for $r \not=2$.
\item[(v)]  $C_7\left\langle i \right\rangle = \{K_1, \overline{K_2},\overline{K_3}\}$
                     and $C_{3r+1}\left\langle i \right\rangle =\{K_1,  \overline{K_2}\}$  for $r \not=2$. 
\end{itemize}
\end{example}

The proof is straightforward and hence we omit it. 
From the above example, one can easily obtain the next result.

\begin{example}\label{disc1}
Let a graph $G$ be an union of $s \geq 2$ paired disjoint cycles $C_{n_1},C_{n_2},..,C_{n_s}$. 
\begin{itemize}
\item[(i)] If $n_i=5$, $i=1,2,..,s$, then $G\left\langle \gamma \right\rangle =  \{\overline{K_1},\overline{K_2},..,\overline{K_{2s}}\}$. 
\item[(ii)] If $n_i=7$, $i=1,2,..,s$, then $G\left\langle \gamma \right\rangle= \{\overline{K_1},\overline{K_2},..,\overline{K_{3s}}\} \cup \{K_2\}$. 
\item[(iii)] If $n_i \not\equiv 1 \pmod 3$ and $n_i \geq 6$ for some $i \in [s]$, then  $G\left\langle \gamma \right\rangle =\{K_1\}$.
\item[(iv)] If $n_i \equiv 1 \pmod 3$  and $n_i \geq 10$ for all $i \in [s]$, then
                    $G\left\langle \gamma \right\rangle = \{K_1, K_2, \overline{K_2}\}$. 
\end{itemize}
\end{example}

Denote by (CEA) the class of all graphs $G$ such that $\gamma(G+e) \not= \gamma(G)$ for all $e \in E(\overline{G})$. 
\begin{example}\label{e2}
Let a noncomplete graph $G$ be in (CEA). 
 It is well known fact that any two nonadjacent vertices of $G$ belong to some $\gamma$-set of $G$ (Sumner and Blitch \cite{sb}).   
In other words,  $G$ is $\{K_1, \overline{K_2}\}$-$\gamma$-excellent graph.  
\end{example}

\begin{proposition}\label{g=b0}
Let $G$ be a graph with $\beta_0(G) = \gamma(G)=s$. 
Then $G$ is $\{\overline{K_1}, \overline{K_2},..,\overline{K_s}\}$-$\gamma$-excellent 
 and \cite{p} $G\left\langle i \right\rangle = G\left\langle \beta_0  \right\rangle= \{\overline{K_1},..,\overline{K_s}\}$. 
\end{proposition}
\begin{proof}
Every independent set of $G$ is a subset of a maximal independent set. 
Since each maximal independent set is always dominating and $\beta_0(G) = \gamma(G)=s$, 
the result immediately follows. 
\end{proof}

The {\em Cartesian product} of two graphs $G$ and $H$ is the graph $G\square H$ whose
vertex set is the Cartesian product of the sets $V (G)$ and $V (H)$. Two vertices
$(u_1, v_1)$ and $(u_2, v_2)$ are adjacent in $G \square H$ precisely when either
 $u_1 = u_2$ and $v_1v_2 \in  E(H)$ or $v_1 = v_2$ and $u_1u_2 \in E(G)$. 
 It is clear from this  definition that $G\square H \simeq H\square G$ and 
if $G$ or $H$ is not connected then $G \square H$ is not connected.

\begin{example}\label{Knn}
Let $G = K_m\square K_n$, $n\geq m \geq 2$. 
Then  $G\left\langle i \right\rangle = G\left\langle \beta_0 \right\rangle= \{\overline{K_1},..,\overline{K_m}\}$.  
If $n > m$, then $G\left\langle \gamma \right\rangle = \{\overline{K_1},..,\overline{K_m}\}$.  
If $n=m$, then 
$G\left\langle \gamma \right\rangle = \{\overline{K_1},..,\overline{K_m}\} \cup \{K_1,K_2,..,K_m\} 
\cup \{K_p \cup \overline{K_q} \mid (p \geq 2) \wedge (q \geq 1) \wedge (p+q \leq m)\}$.
\end{example}

 \begin{proof}
Let $G = K_m\square K_n$, $n\geq m \geq 2$.
 We consider $G$ as an $m \times n$ array of vertices
$\{x_{i,j} \mid (1 \leq i \leq m)  \wedge (1 \leq j \leq n)\}$,   
 where the closed neighborhood of $x_{i,j}$ is the union of 
the sets  $A_i = \{x_{i,1}, x_{i,2},..,x_{i,n}\}$ and $B_j = \{ x_{1,j}, x_{2,j},..,x_{m,j}\}$. 
Then $\left\langle A_i \right\rangle \simeq K_m$ and $\left\langle B_j \right\rangle \simeq K_n$. 
It is well-known that  \cite{gr1} (a) $\gamma(G) = i(G) = \beta_0(G) = m$, 
(b) $A_1,A_2,..,A_m$ are $\gamma$-sets of $G$, 
and if $m=n$, $B_1,B_2,..,B_n$ are also  $\gamma$-sets of $G$. 
Hence, by Proposition \ref{g=b0}, $G$ is $\{\overline{K_1}, \overline{K_2},..,\overline{K_m}\}$-$\gamma$-excellent
and $G\left\langle i \right\rangle = G\left\langle \beta_0 \right\rangle=  \{\overline{K_1},..,\overline{K_m}\}$. 
Suppose that $G$ is $H$-$\gamma$-excellent. 
Then there is a $\gamma$-set $D$   of $G$ such that  $\left\langle D \right\rangle$ 
has an induced subgraph $H_1 \simeq H$.  Assume that $H$ has at least one edge. 
 \medskip

{\em Case} 1: $m < n$. 
 Clearly $|A_i \cap  D|=1$  for all $i=1,2,..,m$. 
Because of symmetry, we assume without loss of generality that 
$D \cap B_j $ is empty for all $j>m$. Define now the set 
$D^t = \{x_{r,s} \mid x_{s,r} \in D\}$.  Since  $H$ is not edgeless, 
 $|D \cap B_j| > 1$ for some $j \leq m$. But  then $|D^t \cap A_j|>1$,   
which means that $D^t$ is not a $\gamma$-set of $G$. 
Since $\left\langle D \right\rangle  \simeq \left\langle D^t \right\rangle$, 
 $G$ is not $H$-$\gamma$-excellent. 
Thus, $G\left\langle \gamma \right\rangle = \{\overline{K_1},..,\overline{K_s}\}$. 
\medskip

{\em Case} 2: $m = n$. 
Obviously in this case  exactly one of  $|A_i \cap  D|=1$  for all $i=1,2,..,m$ 
and $|B_j \cap  D|=1$  for all $j=1,2,..,m$ holds.  Say the first is valid. 
Let $R_1$ be a   $l$-order component   of $\left\langle H \right\rangle$ for some $l \geq 2$. 
For the sake of symmetry, we can  assume  that all  elements of $R_1$ are in $B_1$
 and  $D \subset \cup_{s=1}^pB_s$, where $D \cap B_s$ is not empty for all $s \in [p]$. 
Clearly  $p \leq m-l+1$. 
Suppose that $\left\langle D \right\rangle$  has another nontrivial component. 
Then the difference $m-p$ is not less  than $l$. 
Define the set $D_1 = (D - V(R_1)) \cup \{x_{1,p+1},x_{1,p+2},..,x_{1,p+l}\}$. 
Clearly $D_1$ is not a $\gamma$-set of $G$ and 
$\left\langle D_1 \right\rangle \simeq \left\langle D \right\rangle$. 
Thus $R_1$ is the only nontrivial component of $\left\langle D \right\rangle$. 
Hence $H$ is either a complete graph or a union of complete and edgeless graph. 
Finally, it is easy to see that for each such a graph $H$, 
$G$ is $H$-$\gamma$-excellent. 
\end{proof}

We  need the following  "negative result". 

\begin{theorem}\label{neg}
There is no $P_3$-$\gamma$-excellent graph $G$ with $\gamma(G) = 3$. 
\end{theorem}

\begin{proof}
 Assume that $G$ is a $P_3$-$\gamma$-excellent graph, $\gamma(G) = 3$
 and $x_1,x_2,x_3$ is an induced path in $G$. 
     Since $X = \{x_1,x_2,x_3\}$ is a $\gamma$-set of $G$, 
		there  is $y_i \in pn[x_i, X]$, $i=1,2,3$. 
	Then $\{x_1,x_2,y_2\}$ is a $\gamma$-set of $G$, which implies $y_2y_3 \in E(G)$. 
		But now  no vertex of the induced path $y_2,y_3,x_3$  is adjacent to $x_1$, a contradiction. 
		%the vertices of the path  $y_2,y_3,x_3$ do not dominate $x_1$, a contradiction.
	\end{proof}

\begin{example}\label{compl}
  $\overline{K_3 \square K_n} \left\langle \gamma \right\rangle = 
	\{K_1,K_2,\overline{K_2}, K_1 \cup K_2, \overline{K_3}, K_3\}$ when $n \geq 3$, 
and   $\overline{K_m \square K_n} \left\langle \gamma \right\rangle = 
\{K_1,K_2,\overline{K_2}, K_1 \cup K_2, K_3\}$ when $n \geq m \geq 4$. 
\end{example}
\begin{proof}
First note that $\overline{K_3 \square K_3} \simeq K_3 \square K_3$ and by 
Example \ref{Knn} it immediately follows that $\overline{K_3 \square K_3} \left\langle \gamma \right\rangle = 
	\{K_1,K_2,\overline{K_2}, K_1 \cup K_2, \overline{K_3}, K_3\}$. 
	So, let $n \geq 4$ and $n \geq m \geq 3$. 
	It is well known that \cite{gr1} $\gamma(\overline{K_m \square K_n}) = 3 \leq m = i(\overline{K_m \square K_n})$. 
		Let  us  consider  the graph $G_{m,n} = \overline{K_m \square K_n}$ 
as  a $m \times n$ array  of  vertices $\{a_{i,j} \mid (1 \leq i \leq m )  \wedge  (1 \leq j \leq n)\}$,  
 with  an  adjacency $N(a_{i,j}) = V(G_{m,n} ) - (Y_i \cup Z_j)$,  
where  $Y_i = \cup_{k=1}^{n} \{a_{i,k}\}$ and $Z_j = \cup_{r=1}^{m}\{a_{r,j}\}$. 
Remark now that: 
\begin{itemize}
\item[(a)]    $\left\langle \{a_{i,j}, a_{k,l}, a_{r,s}\} \right\rangle \simeq K_3$
                       if and only if  both $3$-tuples $(i,k,r)$ and $(j,l,s)$ consist of paired distinct integers. 
											 The vertices of each triangle of $G_{m,n}$ form a $\gamma$-set. 
											Every two adjacent vertices  $a_{i,j}$ and $a_{k,l}$ belong to a triangle.
\item[(b)]    All induced subgraphs isomorphic to $K_1 \cup K_2$ are 
                       $\left\langle \{a_{i,j}, a_{k,l}, a_{i,l}\} \right\rangle$ and $\left\langle \{a_{i,j}, a_{k,l}, a_{k,j}\} \right\rangle$, 
											where $i \not=k$ and $j \not= l$. The vertices of each such a subgraph form a $\gamma$-set. 
											Every two vertices belong to an induced subgraph isomorphic to $K_1 \cup K_2$. 
	\item[(c)]  Each $3$-cardinality subset of $Z_j$ is independent and it is not dominating. 
		\end{itemize}

Theorem \ref{neg} together with (a)-(c) immediately lead to the required.
\end{proof}

To continue we need the following theorem and definitions.

\begin{them}\label{lbcp} \cite{ep}
 $\gamma (G \square H) \geq \min \{|V(G)|, |V(H)|\}$ for any two  arbitrary graphs $G$ and $H$.  
\end{them}

A $G$-{\em layer} of the Cartesian product $G \square H$ is the set                               
 $\{(u,y)\mid u \in V(G)\}$,where $y \in V(H)$. 
Analogously an $H$-{\em layer} is the set                                                                              
$\{(x,v) \mid v\in V(H)\}$, 
where $x\in V(G)$. A subgraph of $G \square H$ induced by a $G$-layer or
 an $H$-layer is isomorphic to $G$ or $H$, respectively.

\begin{theorem}\label{existance}
Let $H$ be a connected noncomplete $n$-order graph and $p \geq n \geq 3$.  
If each induced  subgraph of $K_p \square H$ which is  isomorphic to $H$ has as a vertex set some 
 $H$-layer, then  $\gamma(K_p \square H)  = n$ and  $K_p \square H$ is a $H$-$\gamma$-excellent graph. 
\end{theorem}
\begin{proof}
Each $H$-layer of  $K_p \square H$ is a dominating set  of  $K_n \square H$. 
Hence  $\gamma (K_p \square H) \leq |V(H)|=n$. Since $p \geq n$,  by Theorem \ref{lbcp} 
we  have that each $H$-layer is a $\gamma$-set of  $K_p \square H$. 
It remains to note that clearly each vertex of $K_p \square H$ belongs to some $H$-layer. 
\end{proof}

The  next  example serves as an illustration of the above theorem. 

\begin{example}\label{kpcn}
If $p \geq n \geq 5$, then the graph  $K_p \square C_n$ is  $C_n$-$\gamma$-excellent.
\end{example}
\begin{proof}
Let $H$ be an induced subgraph of $K_p \square C_n$ which is isomorphic to $C_r$. 
It is easy to see that if $H$ is not a $C_n$-layer,   then either $r \in \{3,4\}$ or $r \geq n+2$. 
The required immediately follows by Theorem \ref{existance}. 
\end{proof}

\section{Generalized lexicographic product}

Let $G$ be a  graph with vertex set
 $V(G) = \{\textbf{1}, \textbf{2},..,\textbf{n}\}$ and   let 
 $\Phi = (F_1,F_2,..,F_n)$ be  an ordered $n$-tuple of  paired disjoint graphs. 
Denote by $G[\Phi]$ the graph with  vertex set $\cup_{i=1}^n V(F_i)$ 
and edge set defined as follows: (a) $F_1, F_2,.., F_n$ are induced subgraphs of $G[\Phi]$,  
 and  (b) if $x \in V(F_i)$, $y \in V(F_j)$, $i,j \in [n]$ and $i \not= j$, then 
$xy \in E(G[\Phi])$ if and only if \textbf{ij} $\in E(G)$. 
A graph  $G[\Phi]$ is called  the  {\em generalized lexicographic product}  of $G$ and $\Phi$.   
If  $F_i \simeq F$ for every $i=1,2,..,n$, then  $G[\Phi]$
 becomes the standard lexicographic product  $G[F]$. 
Each subset $U = \{u_1,u_2,..,u_n\} \subseteq V(G[\Phi])$ such that 
                         $u_i \in V(F_i)$, for every $i \in [n]$, is called a $G$-layer.
From the definition of $G[\Phi]$ it immediately follow:
\begin{itemize}
\item[(A)] (folklore) $G[\Phi] \simeq G$  if and only if  $G[\Phi] = G[K_1]$. 
                         $G[F] \simeq F$  if and only if $G \simeq  K_1$.
                        If $G$ has at least  two vertices, then $G[\Phi]$ is connected if and only if $G$ is connected. 
												If $G$ is edgeless, then $G[\Phi] = \cup_{i=1}^nF_i$.
                        For any $G$-layer $U = \{u_1,u_2,..,u_n\}$               
                        the bijection  $f \colon V(G)\to U$ defined by  $f(\textbf{i}) = u_i \in V(F_i)$ is an 
										    isomorphism between $G$ and $\left\langle  U \right\rangle$. 
												For any $x \in V(F_i)$ and $y \in V(F_j)$, $i \not= j$, is fulfilled  
												$dist_{G[\Phi]}(x,y) = dist_G(\textbf{i},\textbf{j})$. 
\end{itemize} 
 The equality $dist_{G[\Phi]}(x,y) = dist_G(\textbf{i},\textbf{j})$ 
will be used in the sequel without specific references.

\begin{theorem}\label{three}
Given a graph $G[\Phi]$, where $G$ is  connected of order $n \geq 2$  and 
$|V(F_k)| \geq 3$ for all $k \in [n]$.  Then 
$G[\Phi]\left\langle \gamma \right\rangle = G[\Phi]\left\langle \gamma_r \right\rangle = G[\Phi]\left\langle \gamma^{oc} \right\rangle$ and 
$G[\Phi]\left\langle \gamma_t \right\rangle = G[\Phi]\left\langle \gamma_{tr} \right\rangle = G[\Phi]\left\langle \gamma_t^{oc} \right\rangle$.
If $\gamma(F_k) \geq 3$ for all $k \in [n]$, then 
$G[\Phi]\left\langle \gamma \right\rangle = G[\Phi]\left\langle \gamma_r \right\rangle = G[\Phi]\left\langle \gamma^{oc} \right\rangle = 
G[\Phi]\left\langle \gamma_t \right\rangle = G[\Phi]\left\langle \gamma_{tr} \right\rangle = G[\Phi]\left\langle \gamma_t^{oc} \right\rangle$.
\end{theorem}
\begin{proof}
Let $\mu \in \{\gamma, \gamma_t\}$  and $D$ a $\mu$-set of $G[\Phi]$.  
Assume there is $i \in [n]$ such that $V(F_i) \cap D = \{v_1,v_2,..,v_r\}$, where $r \geq 2$.  
Then clearly for each $\textbf{j} \in N(\textbf{i})$, $V(F_j) \cap D$ is empty and for 
any $u_j \in V(F_j)$ the set $(D-\{v_2,..,v_r\}) \cup \{u_j\}$ is a dominating set of $G[\Phi]$ 
or a total dominating set of $G[\Phi]$ depending on whether $\mu = \gamma$ or 
$\mu = \gamma_t$, respectively.  Hence $r=2$.  Since $G$ is connected of order $n \geq 2$ and 
$|V(F_i)| \geq 3$ for all $i \in [n]$, the graph $\left\langle V(G[\Phi]) - D \right\rangle$ is connected. 
Therefore the first two equality chains are correct. 

Finally, let $D_1$ be a $\gamma$-set of $G[\Phi]$ and  $\gamma(F_k) \geq 3$ for all $k \in [n]$. 
Then clearly for every $i \in [n]$ the sets $D$ and $V(F_i)$  must have no more than one element in common. 
But this immediately implies that $D_1$ is a total dominating set of $G[\Phi]$. 
Thus, the last equality chain holds. 
\end{proof}

\begin{theorem}\label{complete}
Given a graph $G[\Phi]$, where $G$ is  connected of order $n \geq 2$  and  $F_k$ is complete
with $|V(F_k)| \geq 2$ for all $k \in [n]$.  Then $G[\Phi]$ is $\overline{K_s}$-$\gamma$-excellent 
if and only if $G$ is $\overline{K_s}$-$\gamma$-excellent.  
\end{theorem}
\begin{proof} Recall that any $G$-layer of $G[\Phi]$ induces a graph isomorphic to $G$.  We need the following claim. 
\medskip
\\
{\bf Claim 1.}  (i) Each  $\gamma$-set $D$ of $G[\Phi]$ is contained in a $G$-layer of $G[\Phi]$;
                           moreover, $D$ is a $\gamma$-set of each subgraph of $G[\Phi]$ that is induced by 
													a $G$-layer  containing $D$. 
														(ii) If $D^*$ is a $\gamma$-set of some subgraph of $G[\Phi]$ that is induced by 
													      a $G$-layer,  then $D^*$ is a $\gamma$-set of $G[\Phi]$. 
\noindent
\begin{proof}[Proof of Claim 1.]
If $D$ is a  $\gamma$-set of $G[\Phi]$, then since all $F_i$'s are complete $|D \cap V(F_i)| \leq 1$ for all $i \in [n]$. 
But then $D$ is a dominating set of any subgraph of $G[\Phi]$ that is induced by a $G$-layer containing $D$. 
In particular this leads to $\gamma(G[\Phi]) \leq \gamma(G)$.

If $D^*$  is a $\gamma$-set of some subgraph of $G[\Phi]$ that is induced by  a $G$-layer,  
then again by the fact that all $F_i$'s are complete, it follows that $D^*$  is a dominating set of $G[\Phi]$. 
This clearly leads to  $\gamma(G[\Phi]) \geq \gamma(G)$. 

Thus $\gamma(G[\Phi]) = \gamma(G)$ implying the required. 
\end{proof}

$\Leftarrow$ Choose $u \in V(G[\Phi])$ arbitrarily. Then there is a $G$-layer $U$ containing $u$. 
                           Since $G$ is $\overline{K_s}$-$\gamma$-excellent, there is a  $\gamma$-set $D^*$
													of $\left\langle U \right\rangle$ that contains $s$ paired nonadjacent vertices
													one of which is $u$. By  Claim 1, $D^*$ is a 			$\gamma$-set of $G[\Phi]$.  
													
													If $R$ is a $s$-vertex independent set  in $G[\Phi]$, then since all $F_i$'s are complete graphs, 
													$R$ is a subset of some $G$-layer.  The rest is as above.  

$\Rightarrow$ Let $L = \{l_1,l_2,..,l_n\}$ be a $G$-layer of $G[\Phi]$, where $l_i \in V(F_i)$, $i \in [n]$. 
                              Choose $l_r \in L$ arbitrarily. 
                             Since $G[\Phi]$ is $\overline{K_s}$-$\gamma$-excellent, there is an $s$-vertex independent set $I_s$ 
														of $G[\Phi]$ and a $\gamma$-set $D$ of $G[\Phi]$ such that $u \in I_s \subseteq D$. 
														By Claim 1, $D$ is a $\gamma$-set of some subgraph induced by a $G$-layer of $G[\Phi]$.
														%Since all $G$-layers induce  isomorphic subgraphs, it remains to note that we can assume $D \subseteq L$. 
														 Since all $F_i$'s are complete, without loss of generality, we can assume that $D \subseteq L$. 
														
														 Let $R$ be a $s$-vertex independent set of $L$. Then there is a $\gamma$-set $D_1$ of $G[\Phi]$ which 
														has $R$ as a subset. By Claim 1 $D_1$ is a $\gamma$-set of a graph induced by some  													
                            $G$-layer and as above we can assume that $D_1 \subseteq L$. 
\end{proof}

\section{Regular graphs and trees} 
To present the next results on regular graphs, we need the following theorem. 

\begin{them}\label{degree} 
Let $G$ be a $n$-order graph and minimum degree $\delta$. Then 
$\gamma(G) \leq n\delta / (3\delta  -1)$ when $\delta \in \{3,4,5\}$ 
(see \cite{reed}, \cite{sy} and  \cite{xsc}, respectively).     
\end{them}  

For any $5$-regular graph $G$ with $\gamma(G) = 3$, 
the bound stated in Theorem \ref{degree} can be improved by $3$.

\begin{proposition}\label{3-5-10} 
Let $G$ be a $5$-regular graph with $\gamma(G) = 3$.  Then $n \geq 12$. 
\end{proposition}
\begin{proof}
By Theorem \ref{degree} we have  $n \geq 9$.  
  Since  there is no $5$-regular graphs of odd order,  $n \geq 10$ is even. 
	Note that there are exactly sixty $5$-regular graphs of order $10$ \cite{merGT, mer}. 
 Their  adjacency lists can be found in  \cite{mer}. 
A simple verification shows that each of these graphs has  the domination number equals to $2$. 
\end{proof}

\begin{figure}[htbp]
	\centering
		\includegraphics{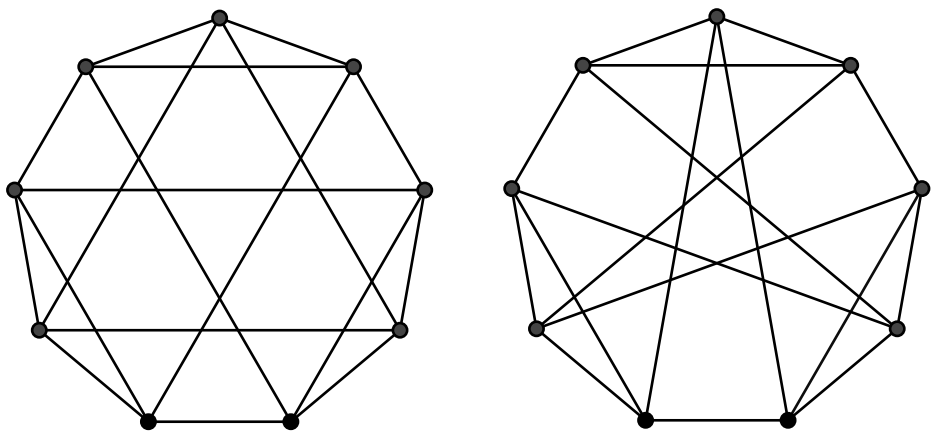}
	\caption{The two $4$-regular $K_3$-$\gamma$-excellent  graphs of order $9$. The graph  on the right is $K_3 \square K_3$. }
	\label{fig: reg9}
\end{figure}

\begin{theorem}\label{reg11}
Let $G$ be a $s$-regular $K_r$-$\gamma$-excellent $n$-order  connected graph 
with $\gamma(G)=r$, where $n > s \geq r \geq 3$. 
Then the following assertions hold. 
\begin{itemize}
\item[(i)]  $n \leq r(s-r+2)$. 
\item[(ii)] If $r=3$, then $s \geq 4$ with equality if and only if $n=9$ and 
                  $G$ is one of the graphs depicted in Fig.\ref{fig: reg9}.
\item[(iii)]  If $r=3$ and  $s=5$, then  $n = 12$.  	
 \end{itemize}
\end{theorem}

\begin{proof}
(i)  Let  $H \simeq K_r$ be a subgraph of $G$. Each vertex of $H$ is adjacent to 
$s-r+1$ vertices outside $V(H)$. Hence  $n \leq r + r(s-r+1) = r(s-r+2)$.

(ii)  Since $r=3$, we have $\gamma(G) =3$ and $n \leq 3s-3$.
     By Theorem \ref{degree} we obtain $8 \leq n$ when $s=3$ 
			and  $9 \leq n$ when $s \geq 4$. 
			Thus $s \geq 4$ and if the equality holds, then $n=9$. 
			There are exactly $16$ $4$-regular graphs of order $9$ \cite{mer}. 
      An immediate verification shows that among them  only  the graphs depicted in Fig.\ref{fig: reg9} 
		  are $K_3$-$\gamma$-excellent. 
			
(iii) 			By (i), $n \leq 12$ and by Proposition \ref{3-5-10} , $n \geq 12$. 
 \end{proof}

Note that  the connected $5$-regular $K_3$-$\gamma$-excellent graph  depicted in Fig.  \ref{fig:12} has  order $12$.

\begin{figure}[htbp]
	\centering
		\includegraphics{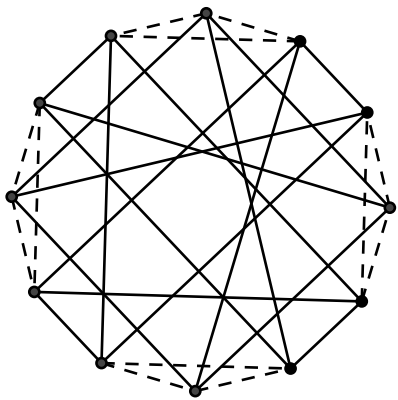}
	\caption{a $5$-regular  $K_3$-$\gamma$-excellent connected graph on $12$ vertices}
	\label{fig:12}
\end{figure}

Now we concentrate on graphs having cut-vertices. 	
				
Let $G_1, G_2,..,G_k$ be pairwise disjoint connected graphs of order 
at least $2$ and $v_i \in V(G_i)$, $i=1,2,..,k$. Then the   {\em coalescence} 
 $(G_1 {\cdot}  G_2 {\cdot}...{\cdot} G_k) (v_1, v_2,.., v_k:v)$
 of $G_1,G_2,...,G_k$  via $v_1, v_2,..,v_k$, 
is the graph obtained from the union of $G_1,G_2,..,G_k$   by identifying $v_1, v_2,.., v_k$
in a vertex labeled $v$. 
If for graphs $G_1, G_2,..,G_k$ is fulfilled $V(G_i) \cap V(G_j) = \{x\}$ when $i,j = 1,2,..,k$ and $i \not= j$, 
then the {\em coalescence}  $(G_1 {\cdot}  G_2 {\cdot}...{\cdot} G_k) (x)$
 of $G_1,G_2,..,G_k$ via $x$ is the union of $G_1, G_2,..,G_k$.

Define  $V^-(G)  = \{x\in V(G) \mid \gamma(G-x) < \gamma(G)\}$ and 
  $V^=(G)  = \{x\in V(G) \mid \gamma(G-x) = \gamma(G)\}$. It is well known that 
$V^-(G)  = \{x\in V(G) \mid \gamma(G-x) +1 = \gamma(G)\}$.  
To continue we need the following result: 

\begin{lema}\label{exc2} \cite{bcd}
Let  $G = (F \cdot H)(x)$.  Then   $x \in V^-(G)$ if and only if $x \in V^-(F) \cap  V^-(G)$. 
Furthermore, if $x \in V^-(G)$, then $\gamma(G) = \gamma(F) + \gamma(H) -1$. 
\end{lema}

\begin{theorem}\label{dwe}
Let $G = (G_1 {\cdot}  G_2 {\cdot}...{\cdot} G_k) (x)$,   $x \in V^-(G)$ and 
$G_i$ is $H$-$\gamma$-exellent, $i=1,2,..,k$, where $H$ is connected and has no cut-vertex. 
Then $G$ is also $H$-$\gamma$-excellent.  
\end{theorem}
\begin{proof} 
Using induction on $k$ we easily obtain from  Lemma \ref{exc2} that  $\{x\} =  V^-(G_1)  \cap V^-(G_2) \cap ...\cap V^-(G_k)$ and 
$\gamma(G) =  \gamma (G_1) + \gamma (G_2) +..+ \gamma (G_k) - k + 1$.
Consider any induced subgraph $R$ of $G$, which  is isomorphic to $H$. 
Since $H$ is connected and without  cut-vertices, $R$ is an induced subgraph of some $G_i$, 
say without loss of generality, $i=1$.  Then there is a $\gamma$-set $D_1$ of $G_1$ 
for which $R$ is an induced subgraph of  $\left\langle D_1 \right\rangle$. 
Let $D_i$ be a $\gamma$-set of $G_i - x$, $i=2,3,..,k$. 
Since  $x \in V^-(G_i)$,  $|D_i| = \gamma(G_j) - 1$. Then 
$D = \cup_{i=1}^kD_i$   is a $\gamma$-set of $G$ and $R$ is an induced subgraph of $\left\langle D \right\rangle$. 
\end{proof}

%To continue we need the following definitions. 
Define a {\em vertex labeling}  of a tree $T$  as a function $S:V (T ) \rightarrow\{0,1\}$. 
A labeled tree $T$ is denoted by a pair $(T , S)$. 
 Let $\textbf{0}_T$ and $\textbf{1}_T$ be the sets of vertices assigned the values $0$ and  $1$, respectively. 
	In a {\em labeled} $1$-{\em corona tree}  	of order at least  four all  its leaves
	are in $\textbf{0}_G$ and all  its support vertices form $\textbf{1}_G$. 

Let $\mathscr{T}$ be the family of labeled trees $(T, S)$ that can be obtained from a sequence of labeled trees 
$\tau: (T^1, S^1), \dots, (T^j, S^j)$, ($j \geq 1$),  
such that  $(T^1, S^1)$ is a  labeled   $1$-corona tree of order at least  four 
and $(T, S) = (T^j, S^j)$, and, if $j \geq 2$,  $(T^{i+1}, S^{i+1})$ can be obtained recursively from $(T^i, S^i)$ 
 by  the following operation:

{\bf Operation}   {$O$}.   The labeled tree $(T^{i+1},S^{i+1})$ is obtained from 
vertex disjoint  $(T^i,S^i)$ and a labeled     $1$-corona tree $G_i$ 
in such a way that  $T^{i+1}= (T^i \cdot G_i)(u,v:u)$, 
where (a) $u \in \textbf{0}_{T^i}$,  $v \in \textbf{0}_{G_i}$ and $u \in \textbf{0}_{T^{i+1}}$, 
and (b) $\textbf{0}_{T^{i+1}} = \textbf{0}_{T^i} \cup \textbf{0}_{G_i}-\{v\}$ 
 and $\textbf{1}_{T^{i+1}} = \textbf{1}_{T^i} \cup \textbf{1}_{G_i}$.

 Now we are in a position to present a (reformulated) 
constructive characterization of $\gamma$-excellent trees.

\begin{them}\label{exc}\cite{sam2}
For any tree $T$ of order at least four the following are equivalent: 
\begin{itemize}
\item[(i)] $T$ is $\gamma$-excellent. 
\item[(ii)] There is labeling $S:V (T ) \rightarrow\{0,1\}$ such that 
                    $(T , S)$ is in $\mathscr{T}$. 
\end{itemize}
Moreover, if $(T , S)$ is in $\mathscr{T}$, then $\textbf{0}_{T}= V^-(T)$, $\textbf{0}_{T}$  is a $\gamma$-set of $T$ 
and $\textbf{1}_{T}= V^=(T)$.  In particular, all leaves of $T$ are in $V^-(T)$. 
\end{them}

Another constructive characterization of the $\gamma$-excellent trees  can be found in \cite{bs}. 
To prove our last result we need the following lemma.

\begin{lemma}\label{bridge} 
Let  $G$ be a connected graph and  $x \in V^-(G)$.  
\begin{itemize}
\item[(i)] If $xy$ is a bridge in $G$, then   no $\gamma$-set of $G$ contains both $x$ and $y$. 
\item[(ii)]If $xy$ and $xz$ are bridges in $G$, then no  $\gamma$-set of $G$ contains both $y$ and $z$. 
 \end{itemize}
\end{lemma}
\begin{proof}
(i)  Clearly, we can consider $G$ as a coalescence $(F\cdot H)(x)$,
       where without loss of generality, $y \in V(F)$ and $x$ is a leaf of $F$.  
       Suppose $D$ is a $\gamma$-set of $G$ and $x,y \in D$. 
       Then $D \cap V(H)$ and $D \cap V(F)$ are dominating sets of $H$ and $F$, respectively.          
        Moreover, since $x$ is a leaf in $F$, $D \cap V(F)$ is not a $\gamma$-set of $F$. 
       Hence $|D| = |D \cap V(H)| + |D \cap V(F)| - 1 \geq \gamma(H) + (\gamma(F) +1) -1 $, 
        a contradiction with Lemma \ref{exc2}. 
				
(ii) Let as in (i), $G=(F\cdot H)(x)$, $y \in V(F)$ and $x$ is a leaf of $F$. 
      Hence $z \in V(H)$.    
       Let  	$D$ be a $\gamma$-set of $G$ and $y,z \in D$. 
			By (i), $x \not\in D$ and then $D \cap V(H)$ and $D \cap V(F)$ are dominating sets of $H$ and $F$, respectively.
		  This implies $|D| = |D \cap V(H)| + |D \cap V(F)| \geq \gamma(H) + \gamma(F)$, 
        a contradiction with Lemma \ref{exc2}. 
\end{proof}

\begin{theorem}
 Let $T$ be a $\gamma$-excellent tree of order at least four.  
\begin{itemize}
\item[(a)] If $T$ has a cut-vertex belonging to $V^-(T)$, then   $T\left\langle \gamma \right\rangle = \{K_1\}$.
\item[(b)] If no cut-vertex of $T$ is in $V^-(T)$, then  $T$ is a $1$-corona tree and 
                    $T\left\langle \gamma \right\rangle = \{\overline{K_1},..,\overline{K_r}\}$, where $2r = |V(T)|$. 
\end{itemize}
 \end{theorem}
 \begin{proof}
Suppose $T$ is $H$-$\gamma$-excellent where $H$ is not edgeless. 
 Let $D$ be a $\gamma$-set of $T$ and $R \simeq H$ be an induced subgraph 
of $\left\langle D \right\rangle$. Choose arbitrarily an edge  $xy$  of $R$.  
Clearly both $x$ and $y$ are not leaves and by Lemma \ref{bridge}, 
neither $x$ nor $y$  is a cut-vertex belonging to $V^-(T)$. 
Hence $x,y \in V^=(T)$, because of Theorem \ref{exc}.  
Now we choose $xy$ so that $x$ is a leaf in $R$.  
By Theorem \ref{exc}, a vertex $y$ has a neighbor $z \in V^-(T)$. 
 Lemma \ref{bridge}  now implies  $N[z] \cap D = \{y\}$. 
But then  the graph $R_x = \left\langle V(R-x) \cup \{z\} \right\rangle$ 
is isomorphic to $R$.   Since  $z \in V^-(T)$ and $yz \in E(T)$, 
 Lemma \ref{bridge} shows that no $\gamma$-set of $T$ contains 
both $y$ and $z$.Thus, we arrive to a contradiction. 
 
Therefore, $T\left\langle \gamma \right\rangle$ contains only edgeless graphs.  
By Theorem \ref{exc} $V^-(T)$ is a $\gamma$-set of $T$. 
Assume first that there is a cut-vertex $x \in V^-(T)$. 
Then for any two neighbors $y$ and $z$ of $x$ the set 
$V_1 = (V^-(T) - \{x\}) \cup \{y,z\}$ is independent of cardinality $\gamma(T) + 1$. 
 Suppose $T$ is $\overline{K_r}$-$\gamma$-excellent for some $r \geq 2$.
 Choose any cardinality $r$ subset $V_1$ of  $ (V^-(T) - \{x\}) \cup \{y,z\}$ 
that contains both $y$ and $z$.  Now by Lemma \ref{bridge}, 
we conclude that no $\gamma$-set of $T$ has $V_1$ as a subset. 
Thus, $T\left\langle \gamma \right\rangle = \{K_1\}$. 

Finally, let $V^-(T)$ contains only leaves. By Theorem \ref{exc}, 
$T$ is a $1$-corona tree.  Clearly  $\gamma(T) = i(T) = \beta_0(T) = r$ 
 and then the required now follows by Proposition \ref{g=b0}. 
 \end{proof}

\section{Open problems and questions}
We conclude the paper by listing some interesting problems and directions for further research.

\begin{itemize}
\item[$\bullet$]  For which ordered pairs $(r, s)$ there are 
                                 $s$-regular $K_r$-excellent graphs of order $r(s-r+2)$ (see Theorem \ref{reg11})?  
                                 Find all  $12$-order $5$-regular  $K_3$-$\gamma$-excellent graphs. 
	\end{itemize}

\begin{itemize}
\item[$\bullet$]  Characterize/describe all graphs $F$ such that
                                 there is no $F$-$\mu$-excellent 	graph $G$ with $\mu(G) = |V(F)|$ (see Observation \ref{mu}). 
                                  Recall that there is no $P_3$-$\gamma$-excellent graph $G$ with $\gamma(G) = 3$ (Theorem \ref{neg}).
  \end{itemize}

\begin{itemize}
\item[$\bullet$] Let $b$ be a positive integer. 
Denote by $\mathscr{A}(\mu,b)$ the class of all $\mu$-excellent connected graphs $G$ 
for which $\mu(G) = b$ and  $|G\left\langle \mu \right\rangle|$ is maximum. 
It might be interesting for the reader to investigate these classes  at least  when $b$ is small. 
Note that we already know that 
 $\mathscr{ A}(\gamma,1)$ consists of all complete graphs, 
and  all connected graphs obtained from  $K_{2n}$, $n \geq 2$,  
      by removing a perfect matching form $\mathscr{ A}(\gamma,2)$ 
      (Example \ref{e12}). 
In addition,  by Example \ref{Knn} we	have $\gamma(K_3\square K_3) = 3$, 
$K_3\square K_3 \left\langle \gamma \right\rangle =
\{K_1, \overline{K_2}, K_2, K_1 \cup K_2, \overline{K_3}, K_3\}$	
and by Theorem \ref{neg} we know that there is no
 $P_3$-$\gamma$-excellent graph $G$ with $\gamma(G) = 3$. 	
Thus, $K_3 \square K_3$ belongs to $\mathscr{ A}(\gamma,3)$
 and  $|K_3\square K_3 \left\langle \gamma \right\rangle|  = 6$.  
Find $\mathscr{ A}(\gamma,3)$.
\end{itemize}

\begin{itemize}
\item[$\bullet$]   Find $T\left\langle \mu \right\rangle$ for each $\mu$-excellent tree  $T$, 
                                  where  $\mu \in \{i,\gamma_t, \gamma_R\}$ and $\gamma_R$
																	stand for the Roman domination number
																	(see \cite{hh}, \cite{h2} and \cite{sam3}, respectively). 
  \end{itemize}

\begin{itemize}
\item[$\bullet$] Find  graphs $H$ such that  each induced  subgraph of $K_p \square H$
                                 which is  isomorphic to $H$ has as a vertex set some  $H$-layer 
																(see  Theorem \ref{existance}). 
  \end{itemize}

\begin{itemize}
\item[$\bullet$]  Characterize/describe all connected $\overline{K_2}$-$\gamma$-excellent graphs $G$ with $\gamma(G)=2$. 
 \end{itemize}


\begin{thebibliography}{30}

\bibitem{berge}  C. Berge, Some common properties for regularizable graphs, edge-critical graphs, and $B$-graphs, in:
Graph Theory and Algorithms (Proc. Symp. Res. Inst. Electr. Comm., Tohoku Univ., Sendai, 1980)
Lecture Notes in Computer Science, Vol. 108 (Springer, Berlin, 1987) 108--123.		

	\bibitem{bcd} R.C. Brigham, P.Z. Chinn, R.D. Dutton, Vertex domination-critical graphs, Networks 18 (1988) 173--179.				
	
	\bibitem{bs} T. Burton, D.P. Sumner, 
$\gamma$-Excellent, critically dominated, end-dominated, and dot-critical trees are equivalent, 	
                          Discrete Mathematics 307 (2007) 683--693		
													
		\bibitem{dz} N. Dean, J. Zito, Well-covered  graphs and   extendability, 
	                         Discrete Math. 126 (1994) 67-80
																								
	\bibitem{ep} M. El-Zahar, C.M. Pareek, Domination number of products of graphs, Ars Combin. 31 (1991), 223--227.			
																									
																																				
\bibitem{fhhhl}  G.~Fricke, T.~Haynes, S.~Hedetniemi, S.~Hedetniemi, R.~Laskar, Excellent trees,
                              Bull. Inst. Comb. Appl.  34(2002), 27--38. 			

													
\bibitem{gr1}   P.J.P. Grobler,	Critical concepts in d	omination, independence and irredundance of graphs, 
                              PhD, 		November, 1998, University of South Africa			
															
																																															
\bibitem {hhs1} T.W. Haynes, S.T. Hedetniemi, P.J. Slater, {\em Fundamentals of Domination in Graphs}, Marcel Dekker, New York, 1998.

						
\bibitem{hh}  T.W. Haynes, M.A. Henning, A characterization of $i$-excellent trees, Discr. Math., 248(2002), 69--77			


\bibitem{h2} M.A. Henning, Total domination excellent trees, Discr. Math., 	263(2003) 93--104			

																
\bibitem{jay} S.R. Jayaram, 		Minimal dominating sets of cardinality two in graphs, 
                                                        Indian J. Pure Appl. Math., 28(1)(1997), 43--46.		
	
																																								
	\bibitem{merGT} M. Meringer, Fast Generation of Regular Graphs and Construction of Cages,
	                                 J. Graph Theory 30, 137--146, 1999.																																							
	
\bibitem{mer} M. Meringer, www.mathe2.uni-bayreuth.de/markus/reggraphs.html\#CRG		

\bibitem{msu}    C.M. Mynhardt, H.C. Stwart, E. Ungerer. Excellent trees and secure domination,
                              Utilitas Math 67(2005) 255--267			

\bibitem{p} M.D. Plummer, Some covering concepts in graphs, J. Combin. Theory 8 (1970) 91--98.

\bibitem{reed} B. Reed, Paths, stars and the number three Combin. Probab. Comput. 5 (1996) 267--276
											
\bibitem{sam2} V. Samodivkin, Domination in graphs,  God. Univ. Arkhit. Stroit. Geod. Sofiya, Svitk II, Mat. Mekh. 39(1996-1997), 111--135 (1999). 

\bibitem{sam3} V. Samodivkin,  Roman domination excellent graphs: trees, Communications Combinatorics and Optimization, 
                                3(1) 2018, 1--24
																
\bibitem{sy} Y. Sohn and X. Yuan Domination in graphs of minimum degree four, J. Korean Math.
                           Soc. 46 (2009), No. 4, pp. 759--773	
													
\bibitem{sry} N. Sridharan and M. Yamuna, Very Excellent Graphs and Rigid Very ExcellentGraphs, 
                         AKCE J. Graphs. Combin., 4(2)(2007),  211--221
												

\bibitem{sb} D.P. Sumner and P. Blitch, Domination critical graphs, J. Combin. theory Ser. B,  34(1983), pp. 65--76.												

\bibitem{xsc} H.M. Xing, L. Sun, X. G. Chen, Domination in graphs of minimum degree five, Graphs
                           Combin. 22 (2006), no. 1, 127--143																																			
	
\bibitem{ys}  M. Yamuna, N. Sridharan, Just excellent graphs, 
                           International Journal of Engineering Science, Advanced Computing and Bio-Technology,
													Vol. 1, No. 3, July –September 2010, pp. 129--136


\end{thebibliography}
\end{document}